\documentclass[11pt]{article}
\usepackage{latexsym, amsmath, amssymb, amsthm, a4, epsfig}
\usepackage{empheq}
\usepackage{graphicx}
\usepackage{color}
\usepackage{comment}
\usepackage{ifthen}
\usepackage{subcaption}
\usepackage{float}
\usepackage{hyperref}
\hypersetup{
setpagesize=false,
 bookmarksnumbered=true,%
 bookmarksopen=true,%
 colorlinks=true,%
 linkcolor=blue,
 citecolor=red,
}

\newtheorem{theorem}{Theorem}[section]
\newtheorem{cor}[theorem]{Corollary}

\newtheorem{prop}[theorem]{Proposition}

\newtheorem{remark}{Remark}

\theoremstyle{definition}

\newcommand{\p}{\partial}

\newcommand{\Rbb}{\mathbb{R}}

\renewcommand{\div}{\mbox{div}\,}

\newcommand{\Ga}{\alpha}

\newcommand{\Gd}{\delta}
\newcommand{\Ge}{\epsilon}

\newcommand{\Gg}{\gamma}

\newcommand{\Gs}{\sigma}

\newcommand{\GD}{\Delta}

\newcommand{\GG}{\Gamma}

\newcommand{\GO}{\Omega}

\def\ol{\overline}

\def\sm{\setminus}

\DeclareMathOperator*{\essinf}{ess inf}

\DeclareMathOperator{\Lip}{Lip}

\numberwithin{equation}{section}
\numberwithin{figure}{section}

\allowdisplaybreaks

\newcommand{\domain}{\Omega}
\newcommand{\bdry}{\partial \domain}
\newcommand{\bdryin}{\Gamma_{-}}
\newcommand{\bdryout}{\Gamma_{+}}

\newcommand{\bdrychar}{\Gamma_0}
\newcommand{\eps}{\epsilon}
\newcommand{\normal}{n}
\newcommand{\adv}{\beta}
\newcommand{\reac}{\mu}
\newcommand{\source}{f}
\newcommand{\weight}{\adv \cdot \normal}
\newcommand{\weightx}{\adv(x) \cdot \normal(x)}
\newcommand{\norm}[1]{\| #1 \|}
\newcommand{\Euclid}[1]{\mathbb{R}^{#1}}
\newcommand{\restric}[2]{\left. #1 \vphantom{\big|} \right|_{#2}}
\newcommand{\sol}{u}
\newcommand{\solER}{u_\epsilon}

\newcommand{\union}{\cup}
\newcommand{\intersec}{\cap}
\newcommand{\abs}[1]{\left | #1 \right |}
\newcommand{\pos}{\reac_0}
\newcommand{\dd}{\adv \cdot \nabla}

\begin{document}

\title{On strong convergence of an elliptic regularization with the Neumann boundary condition applied to a stationary advection equation}

\author{Massaki Imagawa\thanks{Graduate School of Informatics, Kyoto University (m\_imagawa@acs.i.kyoto-u.ac.jp, d.kawagoe@acs.i.kyoto-u.ac.jp)} and Daisuke Kawagoe\footnotemark[1]}

\date{\today}
\maketitle

\begin{abstract}
We consider a boundary value problem of a stationary advection equation with the homogeneous inflow boundary condition in a bounded domain with Lipschitz boundary, and consider its perturbation by $\Ge \GD$, where $\Ge$ is a positive parameter and $\GD$ is the Laplacian. In this article, we show the $L^2$ strong convergence of solutions as the parameter $\Ge$ tends to $0$, and discuss its convergence rates assuming  $H^1$ or $H^2$ regularity for original solutions. A key observation is that the convergence rate depends on the regularity of original solutions and a relation between the boundary and the advection vector field. Some numerical computations support optimality of our convergence estimates.
\end{abstract}
\section{Introduction}

In this article, we discuss $L^2$ strong convergence of an elliptic regularization with the Neumann boundary condition as well as its convergence rates when we apply it to a stationary advection equation.

We consider the following boundary value problems:
\begin{equation} \label{BVP}
\begin{cases}
\adv \cdot \nabla u + \reac u = f &\mbox{ in } \GO,\\
u = 0 &\mbox{ on } \GG_-,
\end{cases}
\end{equation}
and
\begin{equation} \label{MER}
\begin{cases}
-\Ge \GD u + \adv \cdot \nabla u + \reac u = f &\mbox{ in } \GO,\\
u = 0 &\mbox{ on } \GG_-,\\
\dfrac{\p u}{\p n} = 0 &\mbox{ on } \GG_+ \cup \GG_0.
\end{cases}
\end{equation}
Here, $\domain$ is a bounded domain in $\Euclid{d}$ with Lipschitz boundary $\bdry$, and $\adv \in W^{1, \infty}(\GO)^d$ and $\reac \in L^\infty(\GO)$ are functions satisfying
\begin{equation} \label{b0}
\pos := \essinf_{x \in \GO} \left( \reac(x) - \frac{1}{2} \div \adv(x) \right) > 0,
\end{equation}
and $f \in L^2(\GO)$. In the problem \eqref{MER}, $\Ge$ is a positive parameter and $\GD$ is the Laplacian.

We regard the function $\adv$ as a Lipschitz function on $\ol{\GO}$ and decompose the boundary $\p \GO$ into four parts:
\begin{align*}
\GG_+ :=& \{ x \in \p \domain \mid \weightx > 0 \},\\
\GG_- :=& \{ x \in \p \domain \mid \weightx < 0 \},\\
\GG_0 :=& \{ x \in \p \domain \mid \weightx = 0 \},\\
\GG_e :=& \{ x \in \p \domain \mid n(x) \text{ is not defined}. \},
\end{align*}
where $n(x)$ is the outward unit normal vector at $x \in \p \GO$. We call these subsets the outflow boundary, the inflow boundary, the characteristic boundary, and the exceptional set, respectively. Since the boundary $\p \GO$ is Lipschitz, the surface measure of the exceptional set $\GG_e$ is 0.

We ignore the homogeneous Dirichlet boundary condition on the inflow boundary $\GG_-$ in boundary value problems \eqref{BVP} and \eqref{MER} if its surface measure is 0. In the same way, we ignore the homogeneous Neumann boundary condition in \eqref{MER} if the measure of $\bdryout \union \bdrychar$ is $0$.

As we will see later in Section \ref{sec:pre}, under the assumption \eqref{b0}, the boundary value problem \eqref{MER} is well-posed in $H^1_{\GG_-}(\GO)$, where
\[
H^1_{\GG_-}(\GO) := \{ u \in H^1(\GO) \mid u|_{\GG_-} = 0 \}.
\]
We remark that, for the well-posedness of the boundary value problem \eqref{BVP} in
\[
H_{\adv, -}(\GO) := \{ u \in L^2(\GO) \mid \adv \cdot \nabla u \in L^2(\GO), u|_{\GG_-} = 0 \},
\]
some density result is required, which will be mentioned in Section \ref{sec:pre}.

This singular perturbation problem has been addressed at least since 1970 \cite{1970Ba}, where Bardos also considered the perturbation with the homogeneous Dirichlet boundary condition, and made a remarks on a difference of convergences between that case and ours. As far as the authors know it is Livne and Schuss \cite{1973LS} that gave an explicit convergence rate of the perturbation by degenerate elliptic operators for the first time. Since the Laplacian can be regarded as an degenerate elliptic operator, their result covers ours. Goering et al. \cite{GFLRT} discussed pointwise convergence with its convergence rates in the setting of classical solutions by constructing barrier functions near boundary layers, which also yields the convergence rates in the $L^2$ sense. Temam et al. \cite{GJHT} discussed higher order asymptotic expansions of boundary layers and, as a byproduct, they also derived some convergence rates. For the case with the homogeneous Neumann boundary condition, see \cite{JT}. However, these techniques require classical smoothness of solutions. On the other hand, in terms of error analysis on finite element methods, convergence rates in the $L^2$ space are helpful and hence the standard Sobolev spaces $H^m(\GO)$ are natural for the solution space, which was suggested in \cite{Lube}. This motivates us to discuss convergence rates in the $L^2$ framework.

As a related work, Beir\~{a}o da Veiga \cite{1987Veiga1} made use of the perturbed problem with the homogeneous Dirichlet condition in order to capture $W^{1, p}$, $L^p$ and $W^{-1, p}$ solutions to the stationary advection equation. In his setting $\p \GO = \GG_0$ and the boundary condition to the original problem is posed by assuming that the inhomogeneous term $f$ has a compact support.
Bae et al. \cite{BCJK} extended his results to other function spaces.

In this article, we show strong convergence of solutions $\solER$ to \eqref{MER} to the solution $u$ to \eqref{BVP} and give the following estimate under some assumptions for regularity of $u$:
\begin{align}
&\norm{\sol - \solER}_{L^2(\domain)} + \norm{\sol - \solER}_{L^2(\bdryout;\weight)} \leq C \eps^r, \label{ratesL2}\\
&\norm{\nabla(\sol - \solER)}_{L^2(\domain)} \leq C \eps^{r-\frac{1}{2}}, \label{ratesH1}\\
&\norm{\sol - \solER}_{L^2(\bdrychar)} \leq C \eps^{r-\frac{1}{4}}. \label{ratesL2char}
\end{align}
Here, $r$ is a positive constant we will specify later. It is remarkable that the trace of $\solER$ on $\bdry \sm \bdryin$ converges to that of $\sol$, which does not happen in general when we pose the homogeneous Dirichlet boundary condition on $\bdry$.

The organization of the article is the following: Section \ref{sec:pre} is devoted to introducing several propositions and show well-posedness of \eqref{BVP} and \eqref{MER} as preliminaries.

For the case that solution $\sol$ to \eqref{BVP} belongs to $H_{\adv, -}(\domain)$, we can show strong convergence of $(\solER, \restric{\solER}{\bdryout})$ to $(\sol, \restric{u}{\bdryout})$ in $L^2(\domain) \times L^2(\bdryout;\weight)$, where $\solER$ is the solution to $\eqref{MER}$ (Theorem \ref{thm:strong_convergence}).
This will be done in Section $\ref{sec:Ha}$.
When the surface measure of the outflow boundary $\GG_+$ is 0, Theorem \ref{thm:strong_convergence} just implies the strong convergence of $u_\Ge$ to $u$ in $L^2(\GO)$.
We remark that our approach is different from one in \cite{1970Ba}, where the strong convergence of the trace $\restric{\solER}{\bdryout}$ was not discussed explicitly.

In Section $\ref{sec:H1}$, we show that estimates \eqref{ratesL2}, \eqref{ratesH1}, \eqref{ratesL2char} hold with $r=1/2$ when the solution $u$ to \eqref{BVP} belongs to $H^1_{\bdryin}(\domain)$ (Theorem \ref{thm:1/2convergence}). Moreover, we can prove the strong convergence of $u_\Ge$ to $u$ in the $H^1$ sense in this situation (Corollary \ref{cor:H1_strong}).
We remark that, although we assume in Theorem \ref{thm:1/2convergence} that the solution $u$ to the boundary value problem \eqref{BVP} belongs to $H^1_{\GG_-}(\GO)$, it is hard to give a necessary and sufficient condition for its existence. Some sufficient conditions for the existence were provided in \cite{Lube}. As far as the authors know, such a characterization of $H^1$ solutions is not completed yet. For interested readers, see \cite{2017Be, SP}. Instead of posing a sufficient condition for the existence of the $H^1$ solution, we just assume its regularity in order to focus on discussing its convergence rate.

We further discuss the convergence rate assuming that the solution to the boundary value problem \eqref{BVP} belongs to $H^2(\GO) \intersec H_{\adv, -}(\GO)$ in Section \ref{sec:H2}.
The convergence rate depends on the surface measure of $\GG_0$ and the degeneracy of the inner product $\weight$ on $\GG_+$.
In the case that the surface measure of $\bdrychar$ is positive, we have estimates \eqref{ratesL2}, \eqref{ratesH1}, \eqref{ratesL2char} with $r=3/4$ (Theorem \ref{thm:3/4convergence}).
This result agrees with that suggested in \cite{JT}, where they applied the asymptotic expansion in boundary layers appearing near the characteristic boundary $\GG_0$.
When the surface measure of $\bdrychar$ and $\bdryout$ are 0, estimates \eqref{ratesL2}, \eqref{ratesH1}, \eqref{ratesL2char} hold with $r=1$ (Theorem \ref{thm:1convergence0}).
We note that this case was already considered in \cite{GFLRT}.
We also show that the same convergence estimate holds when the condition that the surface measure of $\bdryout$ is $0$ in Theorem \ref{thm:1convergence0} is replaced by $\essinf_{\bdryout} (\weight)>0$ (Theorem \ref{thm:1convergence}).
It is worth mentioning that the one-dimensional case $d = 1$ is covered by Theorem \ref{thm:3/4convergence}, Theorem \ref{thm:1convergence0}, and Theorem \ref{thm:1convergence}.
If we only assume that the surface measure of $\bdrychar$ is 0 in multi-dimensional case $d \geq 2$, the exponent $r$ in estimates \eqref{ratesL2}, \eqref{ratesH1}, \eqref{ratesL2char} varies from $3/4$ to $1$ (Theorem \ref{thm:general_convergence1}).
In this case, it depends on the exponent $\alpha > 0$ satisfying $(\weight)^{-\alpha} \in L^1(\bdryout)$. We emphasize that this is a new observation in this research.

In our results, we only give upper bounds for $L^2$ norms of $u-u_\eps$ of the form \eqref{ratesL2}, \eqref{ratesH1}, \eqref{ratesL2char}. In other words, we only give lower bounds of convergence rates $r$. In section \ref{sec:Num_ex}, some numerical examples are presented in order to verify optimality of our convergence estimates.

\section{Preliminaries} \label{sec:pre}

In this section, we introduce well-posedness results on boundary value problems \eqref{BVP} and \eqref{MER} in suitable spaces. In addition, we will present some inequalities we will use in our discussion.

We define function spaces $H_{\adv, \pm}(\GO)$ by
\[
H_{\adv, \pm}(\GO) := \{ v \in L^2(\GO) \mid \adv \cdot \nabla v \in L^2(\GO), v|_{\GG_\pm} = 0 \},
\]
and introduce the norm $\| \cdot \|_{H_\adv(\GO)}$ defined by
\[
\| v \|_{H_\adv(\GO)}^2 := \| v \|_{L^2(\GO)}^2 + \| \adv \cdot \nabla v \|_{L^2(\GO)}^2.
\]
It is worth mentioning that the traces $v|_{\GG_\pm}$ for a function $v \in L^2(\GO)$ with $\adv \cdot \nabla v \in L^2(\GO)$ does not belong to $L^2(\GG_\pm; |\weight|)$ in general, where $L^2(\GG_\pm; \weight)$ is the function space equipped with the norm $\| \cdot \|_{L^2(\GG_\pm; \weight)}$ defined by
\[
\| u \|_{L^2(\GG_\pm; \weight)}^2 := \int_{\GG_\pm} u^2 \weight\,d\Gs_x.
\]
For example, see \cite{1970Ba}.

We introduce results on the well-posedness of the boundary value problem \eqref{BVP} in $H_{\adv, -}(\GO)$ and boundedness of trace operators. For these results, the density
\begin{equation} \label{density}
H_{\adv, -}(\GO) = \ol{\Lip_{\bdryin}(\ol{\GO})}^{\| \cdot \|_{H_\adv(\GO)}}
\end{equation}
is required, where $\Lip_{\GG_-}(\ol{\GO})$ is the function space consisting of Lipschitz functions on $\ol{\domain}$ which are 0 on $\bdryin$ and $\ol{\Lip_{\bdryin}(\ol{\GO})}^{\| \cdot \|_{H_\adv(\GO)}}$ denotes its closure with respect to the norm $\| \cdot \|_{H_\adv(\GO)}$. We remark that the a sufficient condition for the density \eqref{density} is given in \cite{IK}.

\begin{prop} \label{prop:well-posed0} 
Suppose that the density \eqref{density} holds. Then, the boundary value problem \eqref{BVP} has the unique solution $u \in H_{\adv, -}(\GO)$. Moreover, there exists a positive constant $C$ independent of $f$ such that
\[
\| u \|_{H_\adv (\GO)} \leq C \| f \|_{L^2(\GO)}.
\]
\end{prop}


\begin{prop} \label{prop:trace_adv} 
Suppose that the density \eqref{density} holds. Then, the trace operators $\Gg_\pm: H_{\adv, \mp}(\GO) \to L^2(\GG_\pm; |\weight|)$ are bounded.
\end{prop}


From Proposition \ref{prop:trace_adv}, there exists a positive constant $C$ such that
\begin{equation} \label{bound_trace+}
\| u \|_{L^2(\GO)} + \| u \|_{L^2(\GG_+; \weight)} \leq C \| u \|_{H_\adv(\GO)}
\end{equation}
for all $u \in H_{\adv, -}(\GO)$.

Next we will show the well-posedness of the problem \eqref{MER}. We give a proof here because it contains a key argument in the following sections.

\begin{prop} \label{prop:WP_MER} 
For all $\eps > 0$, the boundary value problem \eqref{MER} has the unique solution $u_\eps \in H^1_{\GG_-}(\GO)$ in the sense that
\begin{equation} \label{eq:weakMER}
\eps \int_\GO \nabla u_\eps \cdot \nabla v\,dx + \int_\GO (\adv \cdot \nabla u_\eps + \reac u_\eps) v \,dx = \int_\GO f v\,dx
\end{equation}
for all $v \in H^1_{\GG_-}(\GO)$.
\end{prop}

\begin{proof}
We define a bilinear form $B_\eps$ on $H^1_{\GG_-}(\GO)^2$ by
\[
B_\eps(u, v) := \eps \int_\GO \nabla u \cdot \nabla v\,dx + \int_\GO (\adv \cdot \nabla u + \reac u) v\,dx
\]
for all $u, v \in H^1_{\GG_-}(\GO)$. It is trivial that the bilinear form $B_\eps$ is continuous on $H^1_{\GG_-}(\GO)^2$. In what follows, we investigate its coercivity on $H^1_{\GG_-}(\GO)^2$.

By integration by parts, we have
\[
\int_\GO (\adv \cdot \nabla u) u\,dx = \int_{\GG_+} u^2 \weight\,d\Gs_x - \int_\GO (\div \adv) u^2\,dx - \int_\GO u (\adv \cdot \nabla u)\,dx
\]
and hence
\[
\int_\GO (\adv \cdot \nabla u) u\,dx = -\frac{1}{2} \int_\GO (\div \adv) u^2\,dx + \frac{1}{2} \int_{\GG_+} u^2 \weight\,d\Gs_x
\]
for all $u \in H^1_{\GG_-}(\GO)$. Thus, we have
\begin{align}
  \begin{split}
    B_\eps(u, u) =& \eps \int_\GO \nabla u \cdot \nabla u\,dx + \int_\GO (\reac - \frac{1}{2} \div \adv) u^2\,dx + \frac{1}{2} \int_{\GG_+} u^2 \weight\,d\Gs_x\\
    \geq& \min \{ \eps, \pos \} \| u \|_{H^1(\GO)}^2 \label{ineq:bilinear}
  \end{split}
\end{align}
for all $u \in H^1_{\GG_-}(\GO)$, which is the coercivity of the bilinear form $B_\eps$ on $u \in H^1_{\GG_-}(\GO)^2$.

The conclusion follows from the Lax-Milgram theorem.
\end{proof}

At the end of this section, we introduce two inequalities for the trace of $H^1$ functions.

\begin{prop} \label{prop:trace_Lp}
Let $\GO$ be a bounded domain of $\Rbb^d$ with Lipschitz boundary $\p \GO$. Then, there exists a constant $C$ such that
\begin{equation} \label{est:trace_L2}
\| u \|_{L^2(\p \GO)} \leq C \left( \eps' \| \nabla u \|_{L^2(\GO)} + \eps'^{-1} \| u \|_{L^2(\GO)} \right)
\end{equation}
for all $u \in H^1(\GO)$ and $0 < \eps' < 1$.
\end{prop}

We can find a proof in \cite{G} for the general $L^p$ case.

\begin{prop} \label{prop:trace_embedding}
Let $\GO$ be a bounded domain in $\Rbb^d$ with Lipschitz boundary, and suppose that $d > 2$ and $2 \leq q \leq 2^* := 2(d - 1)/(d - 2)$. Then, the trace operator $\Gg_{0, q}: H^1(\GO) \to L^q(\p \GO)$ is bounded. If $d = 2$, then the trace operator $\Gg_{0, q}: H^1(\GO) \to L^q(\p \GO)$ is bounded for $p \leq q < \infty$.
\end{prop}

Proposition \ref{prop:trace_embedding} is called the Sobolev trace embedding theorem. We can find its proof in \cite{AF} for the Sobolev trace embedding theorem for general $W^{m, p}$ functions assuming more regularity on the boundary $\p \GO$. We remark that their proof is obviously modified to our cases with Lipschitz boundaries. For example, see \cite{Le}.

\section{$L^2$ strong convergence to $H_{\adv, -}$ solutions} \label{sec:Ha}

In this section, we show strong convergence of the solution $\solER$ to the solution $u$ admitting Theorem \ref{thm:1/2convergence} in Section $\ref{sec:H1}$.

\begin{theorem} \label{thm:strong_convergence}
Suppose that the density \eqref{density} holds. Let $u$ be the solution to the problem \eqref{BVP} in $H_{\adv, -}(\GO)$ and let $u_\Ge$ be the solution to the problem \eqref{MER} in $H^1_{\GG_-}(\GO)$. Then, the sequence $\{ (u_\Ge, u_\Ge|_{\GG_+}) \}$ converges to $(u, u|_{\GG_+})$ strongly in $L^2(\GO) \times L^2(\GG_+; \weight)$ as $\Ge \downarrow 0$.
\end{theorem}

\begin{proof}
Let $u \in H_{\adv, -}(\GO)$ be the unique solution to the problem \eqref{BVP}. Thanks to the density \eqref{density} and the inclusion $\Lip_{\bdryin}(\ol{\domain}) \subset H^1_{\GG_-}(\GO) \subset H_{\adv, -}(\GO)$, $H^1_{\GG_-}(\GO)$ is also dense in $H_{\adv, -}(\GO)$. Thus, for any $\Gd > 0$, there exists a function $u^\Gd \in H^1_{\GG_-}(\GO)$ such that
\begin{equation} \label{approxH1}
\| u - u^\Gd \|_{H_\adv(\GO)} < \Gd.
\end{equation}
We fix $\Gd > 0$ and take $u^\Gd \in H^1_{\GG_-}(\GO)$. Also, we define the function $f^\Gd$ by
\[
f^\Gd := \adv \cdot \nabla u^\Gd + \reac u^\Gd.
\]
Since
\begin{align*}
\| f^\Gd \|_{L^2(\GO)} \leq& \max \{ 1, \| \reac \|_{L^\infty(\GO)} \} \| u^\Gd \|_{H_\adv(\GO)}\\
\leq& \max \{ 1, \| \adv \|_{W^{1, \infty}(\GO)^d} \} \max \{ 1, \| \reac \|_{L^\infty(\GO)} \} \| u^\Gd \|_{H^1(\GO)},
\end{align*}
the function $f^\Gd$ belongs to $L^2(\GO)$. Thus, by Proposition \ref{prop:well-posed0}, the function $u^\Gd$ is the unique solution to the following boundary value problem in $H_{\adv, -}(\GO)$:
  \begin{equation} \label{advection_delta}
  \begin{cases}
  \adv \cdot \nabla u + \reac u = f^\Gd &\mbox{ in } \GO,\\
  u =  0 &\mbox{ on } \GG_-.
  \end{cases}
  \end{equation}
  Then, we have
  \begin{empheq}[left=\empheqlbrace]{alignat*=2}
    &\dd (u-u^\delta) + \reac (u-u^\delta) = f-f^\delta &&\text{ in } \domain, \\
    &u-u^\delta=0 &&\text{ on } \bdryin.
  \end{empheq}
  Thus, we have
  \begin{align} \label{bound_f}
  \| f - f^\Gd \|_{L^2(\GO)} \leq& \| \adv \cdot \nabla (u - u^\Gd) \|_{L^2(\GO)} + \| \reac (u - u^\Gd) \|_{L^2(\GO)} \nonumber\\
  \leq& \max \{ 1, \| \reac \|_{L^\infty(\GO)} \} \| u - u^\Gd \|_{H_\adv(\GO)}.
  \end{align}

  Corresponding to the boundary value problem \eqref{advection_delta}, let us consider the following boundary value problem:
  \begin{equation} \label{ellip_reg_delta}
  \begin{cases}
  -\Ge \GD u_\Ge^\Gd + \adv \cdot \nabla u_\Ge^\Gd + \reac u_\Ge^\Gd = f^\Gd &\mbox{ in } \GO,\\
  u_\Ge^\Gd = 0 &\mbox{ on } \GG_-,\\
  \dfrac{\p u_\Ge^\Gd}{\p n} = 0 &\mbox{ on } \p \GO \setminus \GG_-.
  \end{cases}
  \end{equation}
  Let $u_\Ge$ and $u_\Ge^\Gd$ be solutions to boundary value problems \eqref{MER} and \eqref{ellip_reg_delta}, respectively. Then, we have
\[
\Ge \int_\GO \nabla (u_\Ge - u_\Ge^\Gd) \cdot \nabla v\,dx + \int_\GO (\adv \cdot \nabla (u_\Ge - u_\Ge^\Gd) + \reac (u_\Ge - u_\Ge^\Gd)) v\,dx = \int_\GO (f - f^\Gd) v\,dx
\]
for all $v \in H^1_{\GG_-}(\GO)$. Letting $v = u_\Ge - u_\Ge^\Gd \in H^1_{\GG_-}(\GO)$, we have
\begin{align*}
&\reac_0 \| u_\Ge - u_\Ge^\Gd \|_{L^2(\GO)}^2 + \frac{1}{2} \| u_\Ge - u_\Ge^\Gd \|_{L^2(\GG_+; \weight)}^2\\
\leq& \Ge \| \nabla (u_\Ge - u_\Ge^\Gd) \|_{L^2(\GO)}^2 + \int_\GO (\reac - \frac{1}{2} \div \adv) (u_\Ge - u_\Ge^\Gd)^2\,dx + \frac{1}{2} \| u_\Ge - u_\Ge^\Gd \|_{L^2(\GG_+; \weight)}^2\\
=& \int_\GO (f - f^\Gd) (u_\Ge - u_\Ge^\Gd)\,dx\\
\leq& \| f - f^\Gd \|_{L^2(\GO)} \| u_\Ge - u_\Ge^\Gd \|_{L^2(\GO)}\\
\leq& \frac{1}{2} \reac_0^{-1} \| f - f^\Gd \|_{L^2(\GO)}^2 + \frac{1}{2} \reac_0 \| u_\Ge - u_\Ge^\Gd \|_{L^2(\GO)}^2.
\end{align*}
or
\[
\pos \| u_\Ge - u_\Ge^\Gd \|_{L^2(\GO)}^2 + \| u_\Ge - u_\Ge^\Gd \|_{L^2(\GG_+; \weight)}^2 \leq \pos^{-1} \| f - f^\Gd \|_{L^2(\GO)}^2.
\]
In other words, there exists a positive constant $C$ such that
\begin{equation} \label{est_eps_delta}
\| u_\Ge - u_\Ge^\Gd \|_{L^2(\GO)} + \| u_\Ge - u_\Ge^\Gd \|_{L^2(\GG_+; \weight)} \leq C \| f - f^\Gd \|_{L^2(\GO)}.
\end{equation}

Using inequalities \eqref{bound_trace+} and \eqref{est_eps_delta}, and applying Theorem \ref{thm:1/2convergence}, we have
  \begin{align*}
  &\| u - u_\Ge \|_{L^2(\GO)} + \| u - u_\Ge \|_{L^2(\GG_+; \weight)}\\
  \leq& \| u - u^\Gd \|_{L^2(\GO)} + \| u - u^\Gd \|_{L^2(\GG_+; \weight)} + \| u^\Gd - u_\Ge^\Gd \|_{L^2(\GO)} + \| u^\Gd - u_\Ge^\Gd \|_{L^2(\GG_+; \weight)}\\
  &+ \| u_\Ge^\Gd - u_\Ge \|_{L^2(\GO)} + \| u_\Ge^\Gd - u_\Ge \|_{L^2(\GG_+; \weight)}\\
  \leq& C (\| u - u^\Gd \|_{H_\adv(\GO)} + \Ge^{\frac{1}{2}} \| \nabla u^\Gd \|_{L^2(\GO)} + \| f - f^\Gd \|_{L^2(\GO)} ).
\end{align*}
Here, we take the limit sup as $\Ge \downarrow 0$ to obtain
\[
\limsup_{\Ge \downarrow 0} \left( \| u - u_\Ge \|_{L^2(\GO)} + \| u - u_\Ge \|_{L^2(\GG_+; \weight)} \right) \leq C (\| u - u^\Gd \|_{H_\adv(\GO)} + \| f - f^\Gd \|_{L^2(\GO)} ) \leq C \Gd.
\]
Since the constant $\Gd > 0$ is arbitrary, we have
\[
\lim_{\Ge \downarrow 0} \left( \| u - u_\Ge \|_{L^2(\GO)} + \| u - u_\Ge \|_{L^2(\GG_+; \weight)} \right) = 0,
\]
which implies the strong convergence stated in Theorem \ref{thm:strong_convergence}. This completes the proof.
\end{proof}

\section{Convergence rate to $H^1$ solutions} \label{sec:H1}

In this section, we give estimates \eqref{ratesL2}, \eqref{ratesH1}, \eqref{ratesL2char} with $r=1/2$ assuming that the solution $u$ to the problem \eqref{BVP} belongs to $H^1_{\GG_-}(\GO)$. We rephrase the estimate as follows.

\begin{theorem} \label{thm:1/2convergence} 
Suppose that the solution $u$ to the problem \eqref{BVP} belongs to $H^1_{\GG_-}(\GO)$. Then, there exists a constant $C$ independent of $u$ and $\Ge$ such that
\[
\| u - u_\Ge \|_{L^2(\GO)} + \| u - u_\Ge \|_{L^2(\GG_+; \weight)} \leq C \| \nabla u \|_{L^2(\GO)} \Ge^{\frac{1}{2}}
\]
for all $\Ge > 0$, where $u_\Ge$ is the solution to the problem \eqref{MER}.
\end{theorem}
As corollaries we obtain the following convergence results.

\begin{cor} \label{cor:H1_strong}
Under the assumption of Theorem \ref{thm:1/2convergence}, the sequence of solutions $\{ u_\Ge \}$ converges to $u$ strongly in $H^1_{\GG_-}(\GO)$ as $\Ge \downarrow 0$.
\end{cor}

\begin{cor} \label{cor:1/4convergence_trace}
Suppose that the surface measure of the characteristic boundary $\GG_0$ is positive. Then, under the assumption of Theorem \ref{thm:1/2convergence}, there exists a constant $C$ independent of $u$ and $\Ge$ such that
\[
\| u - u_\Ge \|_{L^2(\GG_0)} \leq C \| \nabla u \|_{L^2(\GO)} \Ge^{\frac{1}{4}}
\]
for all $0 < \Ge < 1$.
\end{cor}

We start from proving Theorem \ref{thm:1/2convergence}. Let $u \in H^1_{\GG_-}(\GO)$ and $u_\Ge \in H^1_{\GG_-}(\GO)$ be solutions to boundary value problems \eqref{BVP} and \eqref{MER}, respectively. Then, they satisfy
\[
\Ge \int_\GO \nabla u \cdot \nabla v\,dx + \int_\GO (\adv \cdot \nabla u + \reac u) v\,dx = \int_\GO fv\,dx + \Ge \int_\GO \nabla u \cdot \nabla v\,dx
\]
and
\[
\Ge \int_\GO \nabla u_\Ge \cdot \nabla v\,dx + \int_\GO (\adv \cdot \nabla u_\Ge + \reac u_\Ge) v\,dx = \int_\GO fv\,dx
\]
for all $v \in H^1_{\GG_-}(\GO)$. By subtracting both sides of the above two equations, we have
\begin{equation} \label{weak_diff}
\Ge \int_\GO \nabla w_\Ge \cdot \nabla v\,dx + \int_\GO (\adv \cdot \nabla w_\Ge + \reac w_\Ge) v\,dx = \Ge \int_\GO \nabla u \cdot \nabla v\,dx
\end{equation}
for all $v \in H^1_{\GG_-}(\GO)$, where $w_\Ge := u - u_\Ge$. We let $v = w_\Ge$. Then, through the same argument as in the proof of Proposition \ref{prop:WP_MER}, we have
\[
\Ge \| \nabla w_\Ge \|_{L^2(\GO)}^2 + \pos \| w_\Ge \|_{L^2(\GO)}^2 + \frac{1}{2} \| w_\Ge \|_{L^2(\GG_+; \weight)}^2 \leq \Ge \| \nabla u \|_{L^2(\GO)} \| \nabla w_\Ge \|_{L^2(\GO)}.
\]

Focusing on the first term of the left hand side, we have
\[
\Ge \| \nabla w_\Ge \|_{L^2(\GO)}^2 \leq \Ge \| \nabla u \|_{L^2(\GO)} \| \nabla w_\Ge \|_{L^2(\GO)},
\]
that is,
\begin{equation} \label{H1uniform_bound}
\| \nabla w_\Ge \|_{L^2(\GO)} \leq \| \nabla u \|_{L^2(\GO)}.
\end{equation}
Then, estimating the second term and the third term, we have
\begin{equation} \label{L2bound_inner}
\pos \| w_\Ge \|_{L^2(\GO)}^2 \leq \Ge \| \nabla u \|_{L^2(\GO)}^2
\end{equation}
and
\begin{equation} \label{L2bound_boundary}
\frac{1}{2} \| w_\Ge \|_{L^2(\GG_+; \weight)}^2 \leq \Ge \| \nabla u \|_{L^2(\GO)}^2.
\end{equation}
Hence, we have
\[
\| w_\Ge \|_{L^2(\GO)} + \| w_\Ge \|_{L^2(\GG_+; \weight)} \leq (\pos^{-\frac{1}{2}} + 2^{\frac{1}{2}}) \| \nabla u \|_{L^2(\GO)} \Ge^{\frac{1}{2}},
\]
which is the desired estimate.

We next give a proof of Corollary \ref{cor:H1_strong}. From \eqref{H1uniform_bound}, we have
\[
\| \nabla u_\Ge \|_{L^2(\GO)} \leq \| \nabla w_\Ge \|_{L^2(\GO)} + \| \nabla u \|_{L^2(\GO)} \leq 2 \| \nabla u \|_{L^2(\GO)}.
\]
Also, from \eqref{L2bound_inner}, we have
\[
\| u_\Ge \|_{L^2(\GO)} \leq \| w_\Ge \|_{L^2(\GO)} + \| u \|_{L^2(\GO)} \leq \pos^{-\frac{1}{2}} \Ge^{\frac{1}{2}} \| \nabla u \|_{L^2(\GO)} + \| u \|_{L^2(\GO)}.
\]
Thus, the family $\{ u_\Ge \}_{0 < \Ge < 1}$ is uniformly bounded in $H^1_{\GG_-}(\GO)$. Thus, there exists a subsequence, which is still denoted by $\{ u_\Ge \}_{0 < \Ge < 1}$, and $u_0 \in H^1_{\GG_-}(\GO)$ such that $\{ u_\Ge \}$ converges to $u_0$ weakly in $H^1_{\GG_-}(\GO)$. Since
\[
\Ge \left| \int_\GO \nabla u_\Ge \cdot \nabla v\,dx \right| \leq \Ge \| \nabla u_\Ge \|_{L^2(\GO)} \| \nabla v \|_{L^2(\GO)} \leq 2 \Ge \| \nabla u \|_{L^2(\GO)} \| \nabla v \|_{L^2(\GO)},
\]
we take $\Ge \downarrow 0$ in \eqref{eq:weakMER} to obtain
\[
\int_\GO (\adv \cdot \nabla u_0 + \reac u_0) v\,dx = \int_\GO f v\,dx
\]
for all $v \in H^1_{\GG_-}(\GO)$, and hence for all $v \in L^2(\GO)$ by the density argument. In other words, $u_0$ is a solution to the problem \eqref{BVP}. By the uniqueness of the solution, we conclude that $u_0 = u$. This implies that the original sequence $\{ u_\Ge \}$ converges to $u$ weakly in $H^1_{\GG_-}(\GO)$.

Going back to the identity \eqref{weak_diff} with $v = w_\Ge$, we see that
\[
\Ge \| \nabla w_\Ge \|_{L^2(\GO)}^2 \leq \Ge \| \nabla w_\Ge \|_{L^2(\GO)}^2 + \int_\GO (\adv \cdot \nabla w_\Ge + \reac w_\Ge) w_\Ge\,dx = \Ge \int_\GO \nabla u \cdot \nabla w_\Ge\,dx,
\]
or
\[
\| \nabla w_\Ge \|_{L^2(\GO)}^2 \leq \int_\GO \nabla u \cdot \nabla w_\Ge\,dx,
\]
Since $\{ u_\Ge \}$ converges to $u$ weakly in $H^1_{\GG_-}(\GO)$, or equivalently $\{ w_\Ge \}$ converges to $0$ weakly in $H^1_{\GG_-}(\GO)$, the right hand side converges to $0$ as $\Ge \downarrow 0$, which implies that $\{ w_\Ge \}$ converges to $0$ strongly in $H^1_{\GG_-}(\GO)$. This complete the proof of Corollary \ref{cor:H1_strong}.

We finally give a proof of Corollary \ref{cor:1/4convergence_trace}. We replace $u$ in \eqref{est:trace_L2} by $w_\Ge$ to obtain
\[
\| w_\Ge \|_{L^2(\GG_0)} \leq \| w_\Ge \|_{L^2(\p \GO)} \leq C \left( \Ge' \| \nabla w_\Ge \|_{L^2(\GO)} + \Ge'^{-1} \| w_\Ge \|_{L^2(\GO)} \right)
\]
for all $0 < \Ge' < 1$. As we saw before, we have $\| \nabla w_\Ge \|_{L^2(\GO)} \leq \| \nabla u \|_{L^2(\GO)}$ and $\| w_\Ge \|_{L^2(\GO)} \leq C \| \nabla u \|_{L^2(\GO)} \Ge^{1/2}$ with some positive constant $C$. Thus we obtain
\[
\| w_\Ge \|_{L^2(\GG_0)} \leq C \| \nabla u \|_{L^2(\GO)} (\Ge' + \Ge'^{-1} \eps^{\frac{1}{2}})
\]
for all $0 < \Ge' < 1$. We notice that the best choice is $\Ge' = \Ge^{1/4}$, which is the estimate in Corollary \ref{cor:1/4convergence_trace}.

\section{Convergence rates to $H^2$ solutions} \label{sec:H2}

In this section, we discuss convergence rates of the elliptic regularization assuming that the solution to the problem \eqref{BVP} belongs to $H^2(\domain) \intersec H_{\adv,-}(\domain)$.

We first give a convergence estimate for a general case.
\begin{theorem} \label{thm:3/4convergence} 
Suppose that the solution $u$ to the problem \eqref{BVP} belongs to $H^2(\domain) \intersec H_{\adv,-}(\domain)$. Also, suppose that the surface measure of the characteristic boundary $\GG_0$ is positive. Then, there exists a constant $C$ independent of $u$ and $\Ge$ such that
\[
\| u - u_\Ge \|_{L^2(\GO)} + \| u - u_\Ge \|_{L^2(\bdryout;\weight)} \leq C \| u \|_{H^2(\GO)} \Ge^{\frac{3}{4}}
\]
and
\[
\| \nabla (u - u_\Ge) \|_{L^2(\GO)} \leq C \| u \|_{H^2(\GO)} \Ge^{\frac{1}{4}}
\]
for all $0 < \Ge < 1$, where $u_{\eps}$ is the solution to the problem \eqref{MER}.
\end{theorem}

\begin{proof}
Assuming that $u \in H^2(\domain) \intersec H_{\adv,-}(\domain)$, we can apply integration by parts to the right hand side of \eqref{weak_diff} to obtain
\begin{equation} \label{int_by_parts}
\Ge \int_\GO \nabla u \cdot \nabla v\,dx = \Ge \int_{\GG \setminus \GG_-} \frac{\p u}{\p n} v\,d\Gs_x - \Ge \int_\GO (\GD u) v\,dx
\end{equation}
for all $v \in H^1_{\GG_-}(\GO)$. Letting $v = w_\Ge$, we have
\begin{align*}
&\Ge \| \nabla w_\Ge \|_{L^2(\GO)}^2 + \pos \| w_\Ge \|_{L^2(\GO)}^2 + \frac{1}{2} \| w_\Ge \|_{L^2(\GG_+; \weight)}^2\\
\leq& \Ge \int_{\GG \setminus \GG_-} \left| \frac{\p u}{\p n} w_\Ge \right|\,d\Gs_x + \Ge \int_\GO |(\GD u) w_\Ge|\,dx\\
\leq& \Ge \left\| \frac{\p u}{\p n} \right\|_{L^2(\p \GO)} \| w_\Ge \|_{L^2(\p \GO)} + \| \GD u \|_{L^2(\GO)} \| w_\Ge \|_{L^2(\GO)}\\
\leq& C \Ge \| u \|_{H^2(\GO)} (\| w_\Ge \|_{L^2(\p \GO)} + \| w_\Ge \|_{L^2(\GO)}).
\end{align*}
For the first term in the right hand side, we use the estimate \eqref{est:trace_L2} with $\Ge' = \Ge^{1/4}$. Applying the Cauchy-Schwarz inequality, we have
\begin{align*}
\Ge \| u \|_{H^2(\GO)} \left( \| w_\Ge \|_{L^2(\p \GO)} + \| w_\Ge \|_{L^2(\GO)} \right) \leq& C \| u \|_{H^2(\GO)} \left( \Ge^{\frac{5}{4}} \| \nabla w_\Ge \|_{L^2(\GO)} + \Ge^{\frac{3}{4}} \| w_\Ge \|_{L^2(\GO)} \right)\\
\leq& C \| u \|_{H^2(\GO)}^2 \Ge^{\frac{3}{2}} + \frac{\Ge}{2} \| \nabla w_\Ge \|_{L^2(\GO)} + \frac{\pos}{2} \| w_\Ge \|_{L^2(\GO)}
\end{align*}
and hence we have
\[
\Ge \| \nabla w_\Ge \|_{L^2(\GO)}^2 + \pos \| w_\Ge \|_{L^2(\GO)}^2 + \| w_\Ge \|_{L^2(\GG_+; \weight)}^2 \leq C \| u \|_{H^2(\GO)}^2 \Ge^{\frac{3}{2}}
\]
for all $0 < \Ge < 1$. Therefore, we have
\[
\| w_\Ge \|_{L^2(\GO)} + \| w_\Ge \|_{L^2(\GG_+; \weight)} \leq C \| u \|_{H^2(\GO)} \Ge^{\frac{3}{4}}
\]
and
\[
\| \nabla w_\Ge \|_{L^2(\GO)} \leq C \| u \|_{H^2(\GO)} \Ge^{\frac{1}{4}}
\]
for all $0 < \Ge < 1$. This completes the proof of Theorem \ref{thm:3/4convergence}.
\end{proof}

From Theorem \ref{thm:3/4convergence} and the estimate \eqref{est:trace_L2} with $\Ge' = \Ge^{1/4}$, we have the following convergence rate on $\bdrychar$.

\begin{cor} \label{cor:1/2convergence_trace} 
Under the assumption of Theorem \ref{thm:3/4convergence}, there exists a constant $C$ independent of $u$ and $\Ge$ such that
\[
\| u - u_\Ge \|_{L^2(\GG_0)} \leq C \| \nabla u \|_{L^2(\GO)} \Ge^{\frac{1}{2}}
\]
for all $0 < \Ge < 1$.
\end{cor}

We next show convergence rates in two good cases.

\begin{theorem} \label{thm:1convergence0} 
Suppose that the solution $u$ to the problem \eqref{BVP} belongs to $H^2(\domain) \intersec H_{\adv,-}(\domain)$. Also, we assume that surface measures of the outflow boundary $\GG_+$ and the characteristic boundary $\GG_0$ are $0$. Then, there exists a positive constant $C$ independent of $u$ and $\Ge$ such that
\[
\| u - u_\Ge \|_{L^2(\GO)} \leq C \| u \|_{H^2(\GO)} \Ge
\]
and
\[
\| \nabla (u - u_\Ge) \|_{L^2(\GO)} \leq C \| u \|_{H^2(\GO)} \Ge^{\frac{1}{2}}
\]
for all $\Ge > 0$, where $u_\Ge$ is the solution to the problem \eqref{MER}.
\end{theorem}

\begin{proof}
Under the assumption on Theorem \ref{thm:1convergence0}, the boundary integral term does not appear in the right hand side of the equation \eqref{int_by_parts} as well as $B_{\eps}$ in Proposition \ref{prop:WP_MER}. Thus, we have
\begin{align*}
\Ge \| \nabla w_\Ge \|_{L^2(\GO)}^2 + \pos \| w_\Ge \|_{L^2(\GO)}^2 \leq& C \Ge \| u \|_{H^2(\GO)} \| w_\Ge \|_{L^2(\GO)}\\
\leq& C \| u \|_{H^2(\GO)}^2 \Ge^2 + \frac{\reac_0}{2} \| w_\Ge \|_{L^2(\GO)}^2,
\end{align*}
or
\[
2 \Ge \| \nabla w_\Ge \|_{L^2(\GO)}^2 + \pos \| w_\Ge \|_{L^2(\GO)}^2  \leq C \| u \|_{H^2(\GO)}^2 \Ge^2,
\]
which implies the estimate in Theorem \ref{thm:1convergence0}.
\end{proof}

\begin{theorem} \label{thm:1convergence} 
Suppose that the solution $u$ to the problem \eqref{BVP} belongs to $H^2(\domain) \intersec H_{\adv,-}(\domain)$. Also, we assume that the surface measure of the characteristic boundary $\GG_0$ is $0$. Furthermore, we assume that there exists a positive constant $n_0$ such that $\weightx \geq n_0$ for all $x \in \GG_+$. Then, there exists a positive constant $C$ independent of $u$ and $\Ge$ such that
\[
\| u - u_\Ge \|_{L^2(\GO)} + \| u - u_\Ge \|_{L^2(\GG_+; \weight)} \leq C \| u \|_{H^2(\GO)} \Ge
\]
and
\[
\| \nabla (u - u_\Ge) \|_{L^2(\GO)} \leq C \| u \|_{H^2(\GO)} \Ge^{\frac{1}{2}}
\]
for all $\Ge > 0$, where $u_\Ge$ is a solution to the problem \eqref{MER}.
\end{theorem}

\begin{proof}
Since $u \in H^2(\GO)$ and the surface measure of $\GG_0$ is $0$, the integration by parts \eqref{int_by_parts} yields
\[
\Ge \int_\GO \nabla u \cdot \nabla v\,dx = \Ge \int_{\GG_+} \frac{\p u}{\p n} v\,d\Gs_x - \Ge \int_\GO (\GD u) v\,dx
\]
for all $v \in H^1_{\GG_-}(\GO)$. Letting $v = w_\Ge$, we have
\begin{align*}
&\Ge \| \nabla w_\Ge \|_{L^2(\GO)}^2 + \pos \| w_\Ge \|_{L^2(\GO)}^2 + \frac{1}{2} \| w_\Ge \|_{L^2(\GG_+; \weight)}^2\\
\leq& \Ge \int_{\GG_+} \left| \frac{\p u}{\p n} \right| |w_\Ge|\,d\Gs_x + \Ge \int_\GO |\GD u| |w_\Ge|\,dx\\
\leq& \Ge \| u \|_{H^2(\GO)} (\| \Gg_1 \| \| w_\Ge \|_{L^2(\GG_+)} + \| w_\Ge \|_{L^2(\GO)}),
\end{align*}
where $\Gg_1: H^2(\GO) \to L^2(\p \GO)$ is the trace operator defined by $\Gg_1 u := \p u / \p n$.

By the assumption on $\GG_+$, we have $\| u \|_{L^2(\GG_+)} \leq n_0^{-1/2} \| u \|_{L^2(\GG_+; \weight)}$ for all $u \in L^2(\GG_+)$. Thus, we have
\begin{align*}
\left( \| w_\Ge \|_{L^2(\GO)} + \| w_\Ge \|_{L^2(\GG_+; \weight)} \right)^2 \leq& C \left\{ \pos \| w_\Ge \|_{L^2(\GO)}^2 + \frac{1}{2} \| w_\Ge \|_{L^2(\GG_+; \weight)}^2 \right\}\\
\leq& C \Ge \| u \|_{H^2(\GO)} (\| w_\Ge \|_{L^2(\GG_+; \weight)} + \| w_\Ge \|_{L^2(\GO)}).
\end{align*}
By dividing the both side by $\| w_\Ge \|_{L^2(\GG_+; \weight)} + \| w_\Ge \|_{L^2(\GO)}$, we obtain the desired estimate.

Revisiting the above estimate, we have
\[
\Ge \| \nabla w_\Ge \|_{L^2(\GO)}^2 \leq \Ge \| u \|_{H^2(\GO)} C (\| w_\Ge \|_{L^2(\GG_+; \weight)} + \| w_\Ge \|_{L^2(\GO)}) \leq C \| u \|_{H^2(\GO)}^2 \Ge^2,
\]
or,
\[
\| \nabla w_\Ge \|_{L^2(\GO)} \leq C \| u \|_{H^2(\GO)} \Ge^{\frac{1}{2}}.
\]
This completes the proof of Theorem \ref{thm:1convergence}.
\end{proof}

\begin{remark} \label{Rem}
Since the surface measure of $\bdrychar$ is $0$, there is no way to discuss the estimate \eqref{ratesL2char} in Theorem \ref{thm:1convergence0} and Theorem \ref{thm:1convergence}. We will not discuss it in Theorem \ref{thm:general_convergence1} too.
\end{remark}

We finally discuss convergence rates when the inner product $\weight$ can degenerate on $\bdryout$.
As we mentioned in the introduction, the one-dimensional case is covered by Theorem \ref{thm:3/4convergence}, Theorem \ref{thm:1convergence0} and Theorem \ref{thm:1convergence}. Thus, in what follows, we discuss the case $d \geq 2$.

\begin{theorem} \label{thm:general_convergence1} 
Suppose that the surface measure of $\GG_0$ is $0$, and that there exists a positive constant $\Ga$ such that the function $(\weight)^{-\Ga}$ is integrable on $\GG_+$. Also suppose that the problem \eqref{BVP} has a solution $u \in H^2(\domain) \intersec H_{\adv,-}(\domain)$. Then, we have the following convergence estimates:
\begin{enumerate}
\item When $d = 2$, for any $q \geq 2$, there exists a constant $C_q$ independent of $u$ and $\Ge$ such that
\[
\| u - u_\Ge \|_{L^2(\GO)} + \| u - u_\Ge \|_{L^2(\bdryout;\weight)} \leq C_q \| u \|_{H^2(\GO)} \Ge^{\min \left\{1, \frac{3}{4} + R(\Ga, q) \right\}}
\]
and
\[
\| \nabla (u - u_\Ge) \|_{L^2(\GO)} \leq C_q \| u \|_{H^2(\GO)} \Ge^{\min \left\{\frac{1}{2}, \frac{1}{4} + R(\Ga, q) \right\}}
\]
for all $0 < \Ge < 1$, where
\[
R(\Ga, q) := \frac{\Ga (q - 2)}{4q}.
\]

\item When $d \geq 3$, there exists a constant $C$ independent of $u$ and $\Ge$ such that
\[
\| u - u_\Ge \|_{L^2(\GO)} + \| u - u_\Ge \|_{L^2(\bdryout;\weight)} \leq C \| u \|_{H^2(\GO)} \Ge^{\min \left\{1, \frac{3}{4} + \frac{\Ga}{4(d - 1)} \right\}}
\]
and
\[
\| \nabla (u - u_\Ge) \|_{L^2(\GO)} \leq C \| u \|_{H^2(\GO)} \Ge^{\min \left\{\frac{1}{2}, \frac{1}{4} + \frac{\Ga}{4(d - 1)} \right\}}
\]
for all $0 < \Ge < 1$.
\end{enumerate}
\end{theorem}

\begin{remark}
Theorem \ref{thm:general_convergence1} claims that $r < (3 + \Ga)/4$ when $d = 2$ and $0 \leq \Ga < 1$.
\end{remark}

\begin{proof}
We note that, by Proposition \ref{prop:trace_embedding}, we have
\[
\left\| \frac{\p u}{\p n} \right\|_{L^q(\p \GO)} \leq \| \Gg_{0, q} \| \sum_{i = 1}^d \| u_{x_i} \|_{H^1(\GO)} \leq C_q \| u \|_{H^2(\GO)}
\]
for all $q \geq 2$ when $d = 2$, and for all $2 \leq q \leq 2(d - 1)/(d - 2)$ when $d \geq 3$.

We first prove the case when $d = 2$. Since the surface measure of $\GG_0$ is 0, as in the proof of Theorem \ref{thm:1convergence}, we have
\begin{align*}
&\Ge \| \nabla w_\Ge \|_{L^2(\GO)}^2 + \pos \| w_\Ge \|_{L^2(\GO)}^2 + \frac{1}{2} \| w_\Ge \|_{L^2(\GG_+; \weight)}^2\\
\leq& \Ge \left\| \frac{\p u}{\p n} \right\|_{L^q(\GG_+)} \| w_\Ge \|_{L^{q'}(\GG_+)} + \Ge \| \GD u \|_{L^2(\GO)} \| w_\Ge \|_{L^2(\GO)}\\
\leq& \Ge \| \Gg_{0, q} \| \| u \|_{H^2(\GO)} \| w_\Ge \|_{L^{q'}(\GG_+)} + \Ge \| \GD u \|_{L^2(\GO)} \| w_\Ge \|_{L^2(\GO)}\\
\leq& C_q \Ge \| u \|_{H^2(\GO)} \left( \| w_\Ge \|_{L^{q'}(\GG_+)} + \| w_\Ge \|_{L^2(\GO)} \right),
\end{align*}
where $q \geq 2$ and $q'$ is the H\"older conjugate of $q$, namely, $1/q + 1/q' = 1$.

We give an estimate for the boundary integral $\| w_\Ge \|_{L^{q'}(\GG_+)}$. By the H\"older inequality, we have
\[
\int_{\GG_+} |w_\Ge|^{q'}\,d\Gs_x \leq \left( \int_{\GG_+} (\weight)^{-\Ga}\,d\Gs_x \right)^{1 - \frac{q'}{2}} \left( \int_{\GG_+} |w_\Ge|^2 (\weight)^{\frac{\Ga(2 - q')}{q'}} \,d\Gs_x \right)^{\frac{q'}{2}},
\]
and hence we have
\[
\| w_\Ge \|_{L^{q'}(\GG_+)} \leq C_q \left( \int_{\GG_+} |w_\Ge|^2 (\weight)^{\frac{\Ga(q - 2)}{q}} \,d\Gs_x \right)^{\frac{1}{2}}.
\]

When $\Ga(q-2)/q \geq 1$, we have
\[
\int_{\GG_+} |w_\Ge|^2 (\weight)^{\frac{\Ga(q - 2)}{q}} \,d\Gs_x \leq \| \adv \|_{L^\infty(\GO)}^{\frac{(\Ga - 1)q - 2 \Ga}{q}} \int_{\GG_+} |w_\Ge|^2 (\weight) \,d\Gs_x,
\]
and hence
\[
\| w_\Ge \|_{L^{q'}(\GG_+)} \leq C_q \| w_\Ge \|_{L^2(\GG_+; \weight)}.
\]
Going back to the original estimate, we have
\begin{align*}
&\Ge \| \nabla w_\Ge \|_{L^2(\GO)}^2 + \pos \| w_\Ge \|_{L^2(\GO)}^2 + \frac{1}{2} \| w_\Ge \|_{L^2(\GG_+; \weight)}^2\\
\leq& C_q \Ge \| u \|_{H^2(\GO)} (\| w_\Ge \|_{L^2(\GG_+; \weight)} + \| w_\Ge \|_{L^2(\GO)}),
\end{align*}
and the same argument as in the proof of Theorem \ref{thm:1convergence} gives
\[
\| w_\Ge \|_{L^2(\GO)} + \| w_\Ge \|_{L^2(\GG_+; \weight)} \leq C_q \| u \|_{H^2(\GO)} \Ge
\]
and
\[
\| \nabla w_\Ge \|_{L^2(\GO)} \leq C_q \| u \|_{H^2(\GO)} \Ge^{\frac{1}{2}}.
\]

In what follows, we assume that $\Ga(q-2)/q < 1$. We take a positive constant $\Gd$ and decompose the above integral into two parts:
\[
\int_{\GG_+} |w_\Ge|^2 (\weight)^{\frac{\Ga(q - 2)}{q}} \,d\Gs_x = \int_{\GG_+ \sm \GG_{+, \Gd}} |w_\Ge|^2 (\weight)^{\frac{\Ga(q - 2)}{q}} \,d\Gs_x + \int_{\GG_{+, \Gd}} |w_\Ge|^2 (\weight)^{\frac{\Ga(q - 2)}{q}} \,d\Gs_x,
\]
where the set $\GG_{+, \Gd}$ is defined by
\[
\GG_{+, \Gd} := \{ x \in \GG_+ \mid \weightx < \Gd \}.
\]
For the first term of the right hand side, we have
\begin{align*}
\int_{\GG_+ \sm \GG_{+, \Gd}} (\weight)^{\frac{\Ga(q - 2)}{q}} |w_\Ge|^2\,d\Gs_x \leq& \Gd^{\frac{\Ga (q - 2)}{q} - 1} \int_{\GG_+ \sm \GG_{+, \Gd}} |w_\Ge|^2 \weight \,d\Gs_x\\
  \leq& \Gd^{\frac{\Ga (q - 2)}{q} - 1} \| w_\Ge \|_{L^2(\GG_+; \weight)}^2.
\end{align*}
For the second term, by the H\"older inequality and \eqref{est:trace_L2}, we have
\begin{align*}
\int_{\GG_{+, \Gd}} |w_\Ge|^2 (\weight)^{\frac{\Ga(q - 2)}{q}} \,d\Gs_x \leq& \Gd^{\frac{\Ga (q - 2)}{q}} \| w_\Ge \|_{L^2(\GG_+)}^2\\
\leq& C_q \Gd^{\frac{\Ga (q - 2)}{q}} \left( \Ge^{\frac{1}{4}} \| \nabla w_\Ge \|_{L^2(\GO)} + \Ge^{-\frac{1}{4}} \| w_\Ge \|_{L^2(\GO)} \right)^2.
\end{align*}
Therefore, we have
\begin{align*}
\| w_\Ge \|_{L^{q'}(\GG_+)} \leq& C_q \left( \int_{\GG_+} |w_\Ge|^2 (\weight)^{\frac{\Ga(q - 2)}{q}} \,d\Gs_x \right)^{\frac{1}{2}}\\
  \leq& C_q \Gd^{\frac{\Ga(q - 2)}{2q}} \left( \Gd^{- \frac{1}{2}} \| w_\Ge \|_{L^2(\GG_+; \weight)} + \Ge^{\frac{1}{4}} \| \nabla w_\Ge \|_{L^2(\GO)} + \Ge^{-\frac{1}{4}} \| w_\Ge \|_{L^2(\GO)} \right).
\end{align*}
Going back to the original estimate, we have
\begin{align*}
&\Ge \| \nabla w_\Ge \|_{L^2(\GO)}^2 + \pos \| w_\Ge \|_{L^2(\GO)}^2 + \frac{1}{2} \| w_\Ge \|_{L^2(\GG_+; \weight)}^2\\
\leq& C_q \Ge \| u \|_{H^2(\GO)} \Gd^{\frac{\Ga(q - 2)}{2q}} \left( \Gd^{- \frac{1}{2}} \| w_\Ge \|_{L^2(\GG_+; \weight)} + \Ge^{\frac{1}{4}} \| \nabla w_\Ge \|_{L^2(\GO)} + \Ge^{-\frac{1}{4}} \| w_\Ge \|_{L^2(\GO)} \right) \\
&+ C_q \Ge \| u \|_{H^2(\GO)} \| w_\Ge \|_{L^2(\GO)}\\
\leq& C_q \| u \|_{H^2(\GO)}^2 \Gd^{\frac{\Ga(q - 2)}{q}} \left(\Ge^2 \Gd^{-1} + \Ge^{\frac{3}{2}} \right)\\
&+ \frac{\Ge}{2} \| \nabla w_\Ge \|_{L^2(\GO)}^2 + \frac{\pos}{2} \| w_\Ge \|_{L^2(\GO)}^2 + \frac{1}{4} \| w_\Ge \|_{L^2(\GG_+; \weight)}^2
\end{align*}
or
\begin{align*}
\Ge \| \nabla w_\Ge \|_{L^2(\GO)}^2 + \pos \| w_\Ge \|_{L^2(\GO)}^2 + \frac{1}{2} \| w_\Ge \|_{L^2(\GG_+; \weight)}^2 \leq C_q \| u \|_{H^2(\GO)}^2 \Gd^{\frac{\Ga(q - 2)}{q}} \left(\Ge^2 \Gd^{-1} + \Ge^{\frac{3}{2}} \right)
\end{align*}
for all $0 < \Ge < 1$. Letting $\Gd = \Ge^{1/2}$, we have
\[
\Ge \| \nabla w_\Ge \|_{L^2(\GO)}^2 + \pos \| w_\Ge \|_{L^2(\GO)}^2 + \frac{1}{2} \| w_\Ge \|_{L^2(\GG_+; \weight)}^2 \leq C_q \Ge^{\frac{3}{2} + \frac{\Ga(q - 2)}{2q}}.
\]
Therefore, we have
\[
\| w_\Ge \|_{L^2(\GO)} + \| w_\Ge \|_{L^2(\GG_+; \weight)} \leq C_q \Ge^{\frac{3}{4} + \frac{\Ga(q - 2)}{4q}}
\]
and
\[
\| \nabla w_\Ge \|_{L^2(\GO)} \leq C_q \Ge^{\frac{1}{4} + \frac{\Ga(q - 2)}{4q}}
\]
for all $q \geq 2$ and $0 < \Ge < 1$.

For the case when $d \geq 3$, we replace the parameter $q$ in the above argument by its upper bound $2(d - 1)/(d - 2)$. Then, we have
\[
\frac{\Ga (q - 2)}{q} = \frac{\Ga}{d - 1}.
\]
Thus, when $\Ga \geq d - 1$ we have
\[
\| w_\Ge \|_{L^2(\GO)} + \| w_\Ge \|_{L^2(\GG_+; \weight)} \leq C \| u \|_{H^2(\GO)} \Ge
\]
and
\[
\| \nabla w_\Ge \|_{L^2(\GO)} \leq C \| u \|_{H^2(\GO)} \Ge^{\frac{1}{2}}
\]
for all $0 < \Ge < 1$, and when $\Ga < d - 1$ we have
\[
\| w_\Ge \|_{L^2(\GO)} + \| w_\Ge \|_{L^2(\GG_+; \weight)} \leq C \Ge^{\frac{3}{4} + \frac{\Ga}{4(d - 1)}}
\]
and
\[
\| \nabla w_\Ge \|_{L^2(\GO)} \leq C \Ge^{\frac{1}{4} + \frac{\Ga}{4(d - 1)}}
\]
for all $0 < \Ge < 1$. Therefore, Theorem \ref{thm:general_convergence1} is proved.
\end{proof}

\section{Numerical experiments} \label{sec:Num_ex}

In this section, we present some numerical experiments in the two-dimensional case $d=2$ to verify optimality of convergence rates we have obtained.
In order to compute the convergence rate, we give
functions $\adv$, $\reac$ and $\source$ so that the exact solution $u$ to $\eqref{BVP}$ is known. For each $\eps=1.6^{-k}$, $k=0,1,\ldots, 14$ ($1.6^{-14} \approx 0.0014$), we solve $\eqref{MER}$ by using the finite  element solver FreeFem \cite{H} and obtain a numerical approximation of $L^2$ norms of $\sol-\solER$. Using the least square fitting we numerically estimate the convergence rate and compare it to our theoretical results. In these numerical experiments, the domain $\domain$ is fixed to be the unit square $(0,1)^2$.

For the numerical computation of $\solER$, we use a $\mathrm{P}1$ finite element space associated with a triangulation $\mathcal{T}_h$ whose discretization parameter $h:=\max_{K \in \mathcal{T}_h} \mathrm{diam}(K)$ is about $0.002$.
We remark that our experiments always contain discretization error. In order to lesser their effects, we use a refined mesh whose discretization parameter satisfies $h \approx 0.001$ when we compute $L^2$ errors.

\subsection{Example 1: corresponding to Theorem \ref{thm:1/2convergence}}

Recall that, if the solution $\sol$ to the problem \eqref{BVP} belongs to $H^1_{\bdryin}(\domain)$,
then it follows from Theorem \ref{thm:1/2convergence},  Corollary \ref{cor:H1_strong} and Corollary \ref{cor:1/4convergence_trace} that
\begin{align*}
  &\norm{\sol - \solER}_{L^2(\domain)} + \norm{\sol - \solER}_{L^2(\bdryout;\weight)} \leq C \eps^{\frac{1}{2}}, \\
  &\norm{\nabla(\sol - \solER)}_{L^2(\domain)} \leq C,\\
  &\norm{\sol - \solER}_{L^2(\bdrychar)} \leq C \eps^{\frac{1}{4}}.
\end{align*}
Let $s$ be a positive constant. We take
\begin{gather*}
  \adv(x_1,x_2)=(x_1, 1), \quad \reac(x_1,x_2)=1, \quad
  f(x_1, x_2)=(s+1)x_1^sx_2+x_1^s+x_2+1,
\end{gather*}
and the exact solution to $\eqref{BVP}$ is
\begin{gather*}
  u(x_1, x_2) = (1+x_1^s)x_2,
\end{gather*}
which belongs to $H_{\adv,-}(\domain)$ for all $s>0$ and to $H^1_{\bdryin}(\domain)$ if and only if $s>1/2$.
In particular, we take $s \in (1/2, 1)$, then $u$ is in $H^1_{\bdryin}(\domain) \setminus H^2(\domain)$.
The $L^2$ errors of $u-\solER$ for $s=0.51$ are shown in Figure $\ref{ex:1a}$.
These graphs almost agree with our estimates.

\begin{figure}[H]
    \begin{tabular}{cc}
      \begin{minipage}[t]{0.49\hsize}
        \centering
        \includegraphics[keepaspectratio, scale=0.55]{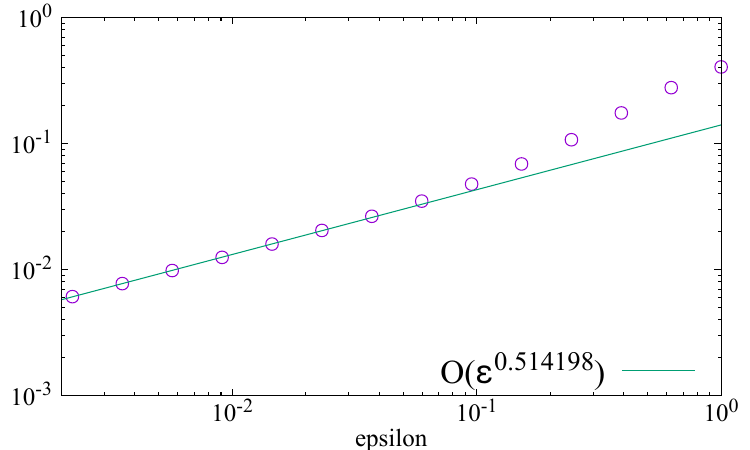}
        \subcaption{$\norm{\sol-\solER}_{L^2(\domain)}$}
        \label{ex:1a_a}
      \end{minipage} &
      \begin{minipage}[t]{0.49\hsize}
        \centering
        \includegraphics[keepaspectratio, scale=0.55]{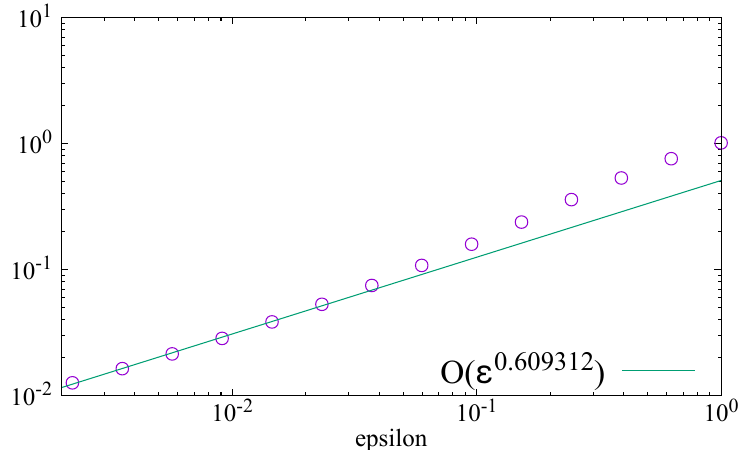}
        \subcaption{$\norm{\sol-\solER}_{L^2(\bdryout;\weight)}$}
        \label{ex:1a_b}
      \end{minipage} \\

      \begin{minipage}[t]{0.49\hsize}
        \centering
        \includegraphics[keepaspectratio, scale=0.55]{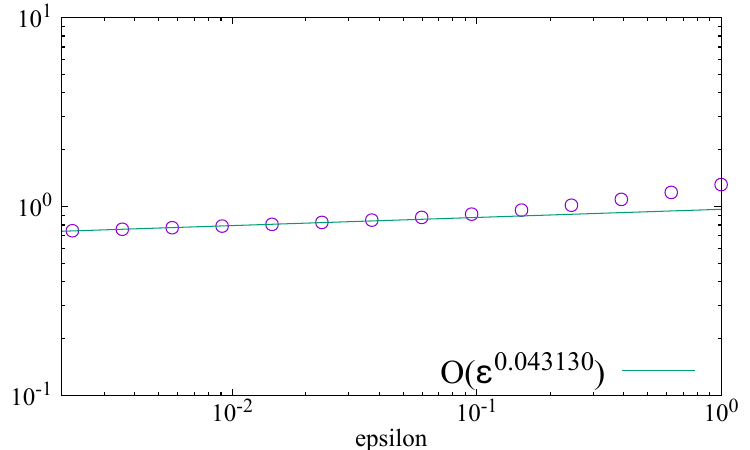}
        \subcaption{$\norm{\nabla(\sol-\solER)}_{L^2(\domain)}$}
        \label{ex:1a_d}
      \end{minipage} &
      \begin{minipage}[t]{0.49\hsize}
        \centering
        \includegraphics[keepaspectratio, scale=0.55]{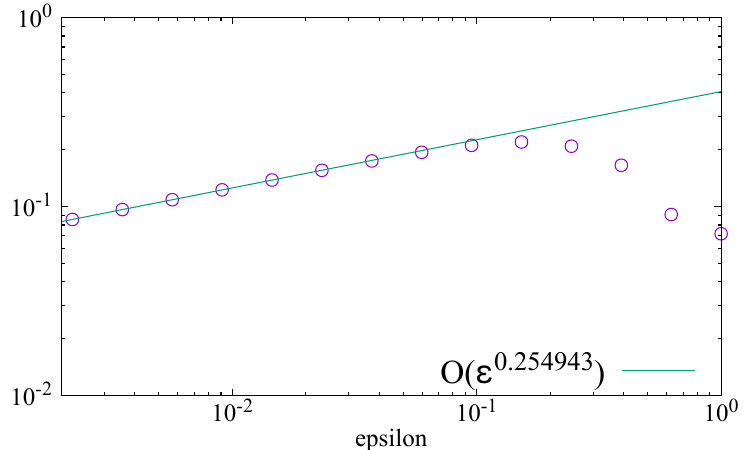}
        \subcaption{$\norm{\sol-\solER}_{L^2(\bdrychar)}$}
        \label{ex:1a_c}
      \end{minipage}

    \end{tabular}
    \caption{Norms of $\sol - \solER$ vs $\eps$ in log-log scale. For the lease square fitting data corresponding to $\eps=1.6^{-k}$, $8 \leq k \leq 14$ are used. }
    \label{ex:1a}
  \end{figure}


If we take $s \in (0,1/2)$, then $u$ no longer belongs to $H^1_{\bdryin}(\domain)$.
The $L^2$ errors of $\sol-\solER$ for $s=0.3$ are shown in Figure $\ref{ex:1b}$. These graphs insist that convergence rate would become worse if $u$ would not have $H^1$ regularity.

\begin{figure}[H]

      \begin{minipage}[t]{0.49\hsize}
        \centering
        \includegraphics[keepaspectratio, scale=0.55]{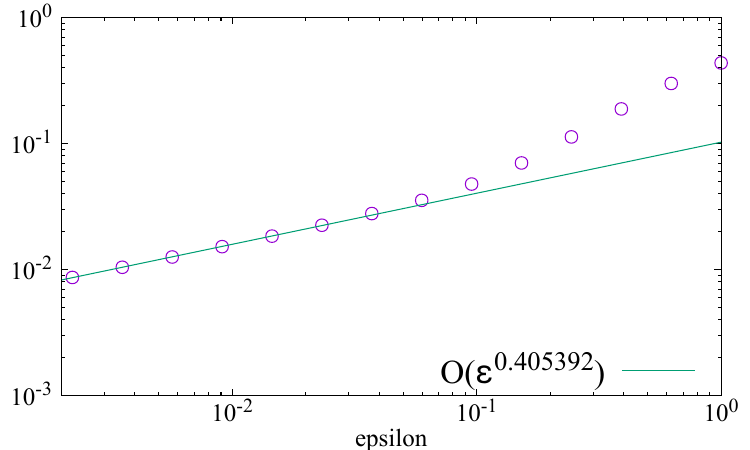}
        \subcaption{$\norm{\sol-\solER}_{L^2(\domain)}$}
        \label{ex:1b_a}
      \end{minipage}
      \begin{minipage}[t]{0.49\hsize}
        \centering
        \includegraphics[keepaspectratio, scale=0.55]{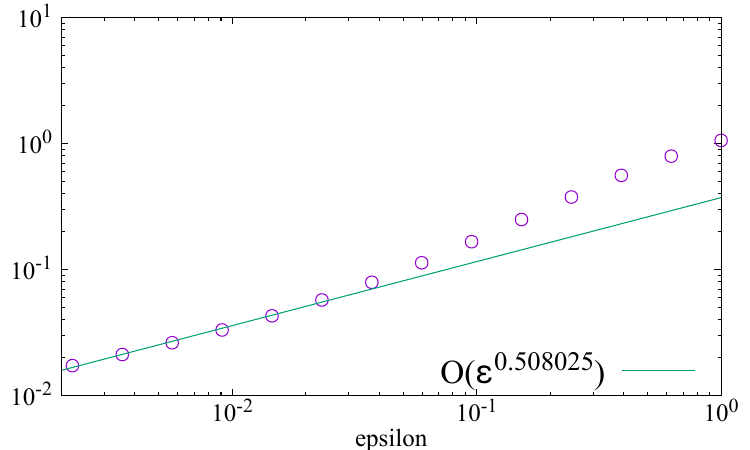}
        \subcaption{$\norm{\sol-\solER}_{L^2(\bdryout;\weight)}$}
        \label{ex:1b_b}
      \end{minipage} \\

      \centering
      \begin{minipage}[t]{0.49\hsize}
        \centering
        \includegraphics[keepaspectratio, scale=0.55]{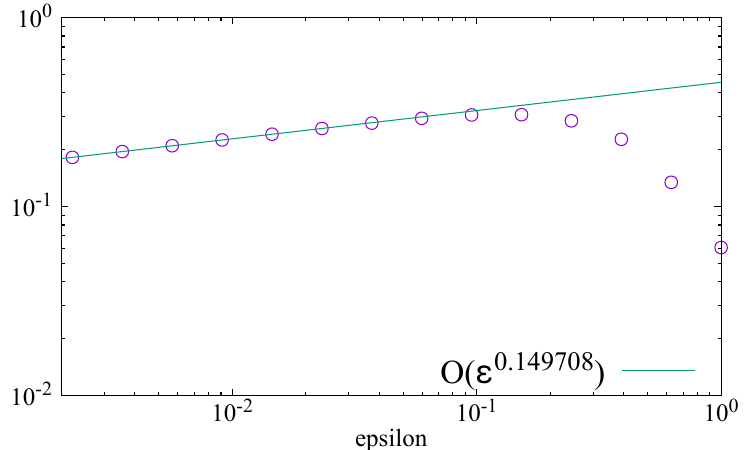}
        \subcaption{$\norm{\sol-\solER}_{L^2(\bdrychar)}$}
        \label{ex:1b_c}
      \end{minipage} 

    \caption{Norms of $\sol - \solER$ vs $\eps$ in log-log scale. For the lease square fitting data corresponding to $\eps=1.6^{-k}$, $8 \leq k \leq 14$ are used. }
    \label{ex:1b}
  \end{figure}


\subsection{Example 2: corresponding to Theorem $\ref{thm:3/4convergence}$}

We proved in Theorem $\ref{thm:3/4convergence}$ and Corollary \ref{cor:1/2convergence_trace} that, if $u \in H^2(\domain) \intersec H_{\adv,-}(\domain)$ solves the problem \eqref{BVP} and $|\bdrychar|>0$, where $\abs{A}$ denotes the surface measure of the set $A$, then it holds that
\begin{align*}
  &\norm{\sol - \solER}_{L^2(\domain)} + \norm{\sol - \solER}_{L^2(\bdryout;\weight)} \leq C \eps^{\frac{3}{4}}, \\
  &\norm{\nabla(\sol - \solER)}_{L^2(\domain)} \leq C \eps^{\frac{1}{4}},\\
  &\norm{\sol - \solER}_{L^2(\bdrychar)} \leq C \eps^{\frac{1}{2}}.
\end{align*}

We take
\begin{gather*}
  \adv(x_1,x_2)=(1, 0), \quad \reac(x_1,x_2)=1, \quad \source(x_1, x_2) = x_1x_2+x_2
\end{gather*}
so that the solution to $\eqref{BVP}$ is
\begin{gather*}
  u(x_1,x_2)=x_1x_2,
\end{gather*}
which belongs to $H^2(\domain) \intersec H_{\adv,-}(\domain)$.
In Figure $\ref{ex:2}$ we show our results of computation.

\begin{figure}[H]
    \begin{tabular}{cc}
      \begin{minipage}[t]{0.49\hsize}
        \centering
        \includegraphics[keepaspectratio, scale=0.55]{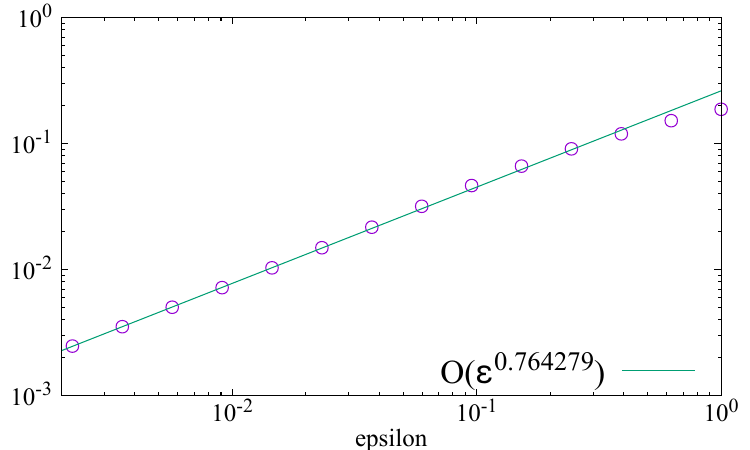}
        \subcaption{$\norm{\sol-\solER}_{L^2(\domain)}$}
        \label{ex:2_a}
      \end{minipage} &
      \begin{minipage}[t]{0.49\hsize}
        \centering
        \includegraphics[keepaspectratio, scale=0.55]{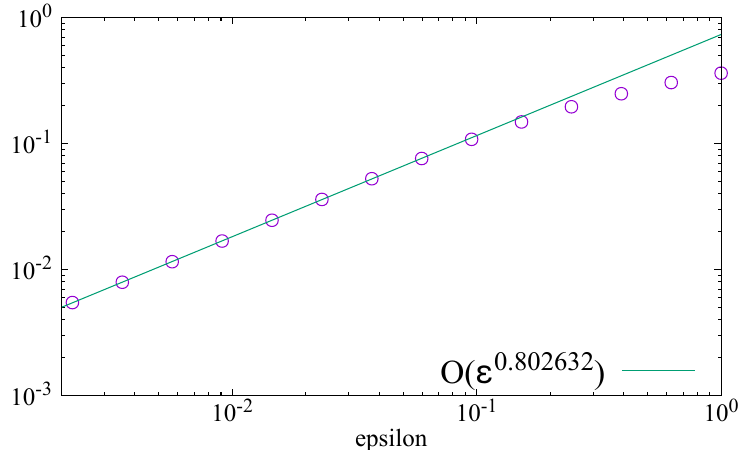}
        \subcaption{$\norm{\sol-\solER}_{L^2(\bdryout;\weight)}$}
        \label{ex:2_b}
      \end{minipage} \\

      \begin{minipage}[t]{0.49\hsize}
        \centering
        \includegraphics[keepaspectratio, scale=0.55]{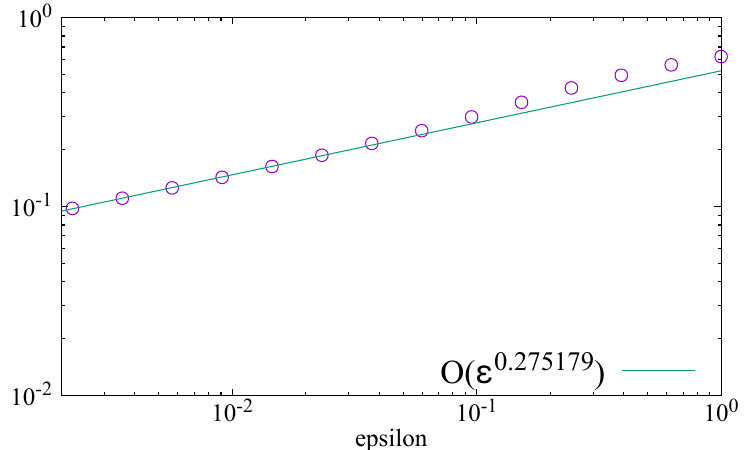}
        \subcaption{$\norm{\nabla(\sol-\solER)}_{L^2(\domain)}$}
        \label{ex:2_d}
      \end{minipage} &
      \begin{minipage}[t]{0.49\hsize}
        \centering
        \includegraphics[keepaspectratio, scale=0.55]{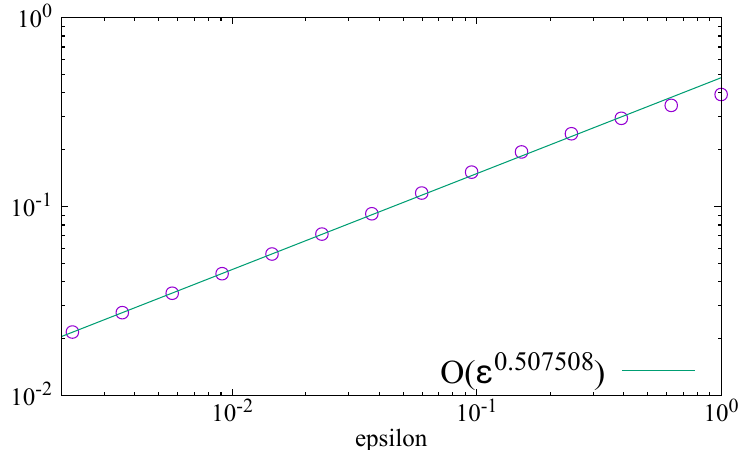}
        \subcaption{$\norm{\sol-\solER}_{L^2(\bdrychar)}$}
        \label{ex:2_c}
      \end{minipage}

    \end{tabular}
    \caption{Norms of $\sol - \solER$ vs $\eps$ in log-log scale. For the lease square fitting data corresponding to $\eps=1.6^{-k}$, $8 \leq k \leq 14$ are used. }
    \label{ex:2}
  \end{figure}


\subsection{Example 3: corresponding to Theorem $\ref{thm:1convergence}$}

By Theorem $\ref{thm:1convergence}$, if $u \in H^2(\GO) \intersec H_{\adv, -}(\GO)$ solves the problem \eqref{BVP}, $\abs{\bdrychar}=0$ and $\weight$ is uniformly positive on $\bdryout$, then it holds that
\begin{align*}
  &\norm{\sol - \solER}_{L^2(\domain)} + \norm{\sol - \solER}_{L^2(\bdryout;\weight)} \leq C \eps, \\
  &\norm{\nabla(\sol - \solER)}_{L^2(\domain)} \leq C\eps^{\frac{1}{2}}.
\end{align*}
In this example, we take
\begin{gather*}
  \adv(x_1,x_2)=(1, 1), \quad \reac(x_1,x_2)=1,\quad \source(x_1, x_2) = (4x_2+1)\sin(4x_1)+\cos(4x_1),
\end{gather*}
which satisfy $\eqref{b0}$ and $f \in L^2(\domain)$. The exact solution to $\eqref{BVP}$ is
\begin{gather*}
  u(x_1,x_2)=x_2\sin(4x_1),
\end{gather*}
which belongs to $H^2(\domain) \intersec H_{\adv,-}(\domain)$. Figure $\ref{ex:3}$ shows that the rate of convergence of $L^2$ errors is approximately equal to $1$, which is expected by the theorem.

\begin{figure}[H]
  \begin{minipage}[t]{0.49\hsize}
    \centering
    \includegraphics[keepaspectratio, scale=0.55]{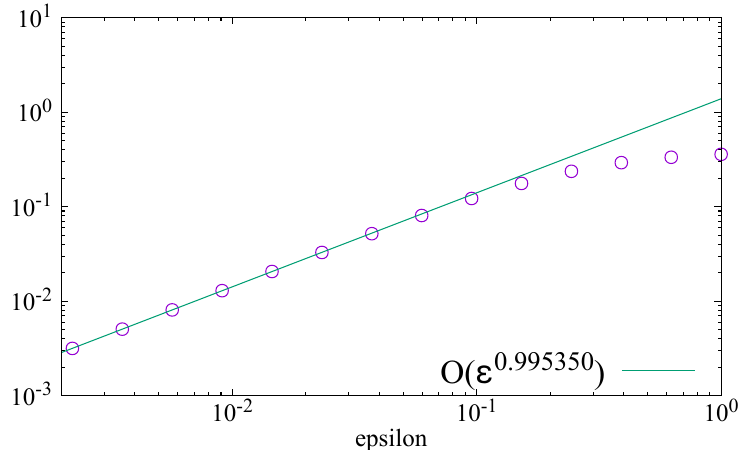}
    \subcaption{$\norm{\sol-\solER}_{L^2(\domain)}$}
    \label{ex:3_a}
  \end{minipage}
  \begin{minipage}[t]{0.49\hsize}
    \centering
    \includegraphics[keepaspectratio, scale=0.55]{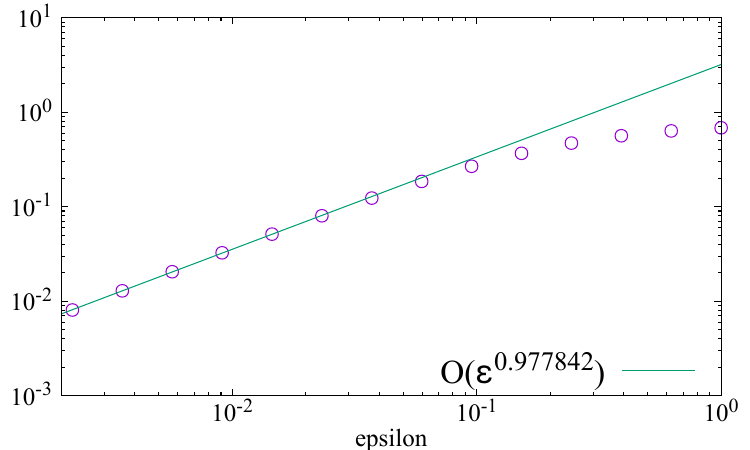}
    \subcaption{$\norm{\sol-\solER}_{L^2(\bdryout;\weight)}$}
    \label{ex:3_b}
  \end{minipage}
  \\
  \centering
  \begin{minipage}[t]{0.49\hsize}
    \includegraphics[keepaspectratio, scale=0.55]{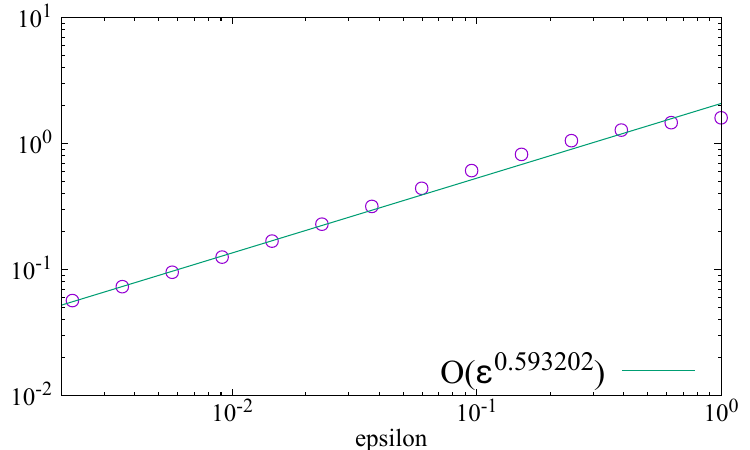}
    \subcaption{$\norm{\nabla(\sol-\solER)}_{L^2(\domain)}$}
    \label{ex:3_d}
  \end{minipage}
    \caption{Norms of $\sol - \solER$ vs $\eps$ in log-log scale. For the lease square fitting data corresponding to $\eps=1.6^{-k}$, $8 \leq k \leq 14$ are used. }
    \label{ex:3}
  \end{figure}


\subsection{Example 4: corresponding to Theorem $\ref{thm:general_convergence1}$}
Theorem $\ref{thm:general_convergence1}$ says that if $\sol \in H^2(\GO) \intersec H_{\adv, -}(\GO)$ solves the problem \eqref{BVP} and $\abs{\bdrychar}=0$, then the convergence rate depends on the maximum exponent $\alpha$ satisfying $(\weight)^{-\alpha} \in L^1(\bdryout)$.
In order to check this, we consider the following family of vector fields parametrized by $s>0$:
\begin{gather*}
  \adv_s =(1-x_1+(1-x_2)^s, 1+x_2).
\end{gather*}
For each $\adv_s$, $\bdryout=\bdry_t \union \bdry_r$, where $\bdry_t=(0,1) \times \{1\}$ and $\bdry_r:=\{1\} \times (0,1)$.
On $\bdry_t$ and $\bdry_r$, we have
\begin{align*}
  \int_{\bdry_t}(\weight)^{-\alpha} d\sigma_x&=\int_0^1 (\adv_s(x_2,1) \cdot \normal(x_1, 1))^{-\alpha} dx_1 = 2^{-\alpha}, \\
  \int_{\bdry_r}(\weight)^{-\alpha} d\sigma_x &= \int_0^1(\adv_s(1, x_2) \cdot \normal(1,x_2))^{-\alpha}dx_2 =
  \int_0^1 (1-x_2)^{-s\alpha} dx_2,
\end{align*}
respectively.
So $(n \cdot \adv_s)^{-\alpha} \in L^1(\bdryout)$ if and only if $s \alpha <1$.
Now we take
\begin{gather*}
  \sol(x_1, x_2)=(e^{x_1}-1)\sin(x_2), \quad \reac=1,
\end{gather*}
and set $f_s:=\adv_s \cdot \nabla \sol + \reac \sol$ for each $s>0$.
We numerically compute the convergence rate of $\norm{u-\solER}_{L^2(\domain)}$ for $f=f_s$ by least square fitting.
In Figure $\ref{ex:4}$ we show the relatioin between the convergence rate and $\alpha=1/s$,
which is also expected by the theorem.
\begin{figure}[H]
  \centering
  \includegraphics[width=0.7\linewidth]{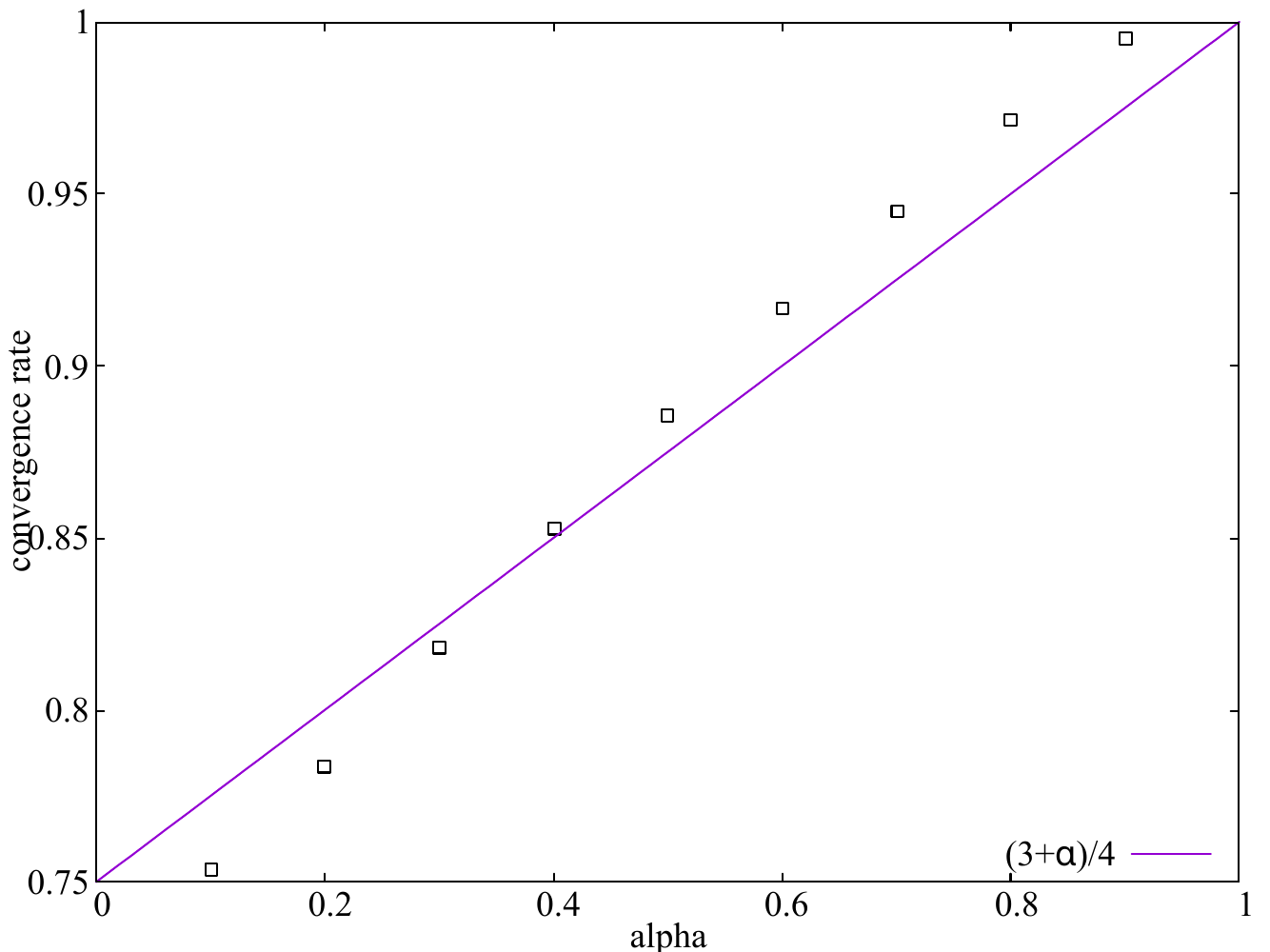}
  \caption{Convergence rate of $\norm{u-\solER}_{L^2(\domain)}$ vs $\alpha$. The line in above graph is $(3+\alpha)/4$, which is the expected convergence rate. } \label{ex:4}
\end{figure}

\section*{Acknowledgement}
The authors thank Professor Emeritus Yuusuke Iso for suggesting this problem. They also would like to express their gratitude to Professor Emeritus Gert Lube for his introducing some previous works. They thank Professor Hiroshi Fujiwara for giving helpful advice from the viewpoint of numerical analysis and numerical computation. This work was supported by JST Grant Number JPMJFS2123 and by JSPS KAKENHI Grant Numbers JP21H00999, JP20K14344.

\end{document}